\documentclass[a4paper]{amsart}
\usepackage{color}
\usepackage[latin1]{inputenc}
\usepackage[T1]{fontenc}
\usepackage{amsfonts}
\usepackage{amssymb}
\usepackage{amsmath}
\usepackage{amsthm}
\usepackage{stmaryrd}
\usepackage{eufrak}
\usepackage{graphicx,type1cm,eso-pic,color}
\usepackage{pstricks,pst-plot,pstricks-add}
\usepackage{float}
\newtheorem{theorem}{Theorem}[section]

\newtheorem{proposition}[theorem]{Proposition}

\newtheorem{definition}[theorem]{Definition}
\newtheorem{assumption}[theorem]{Assumption}

\begin{document}
\setlength\arraycolsep{2pt}
\title{On a Constrained Fractional Stochastic Volatility Model}
\author{Nicolas MARIE}
\address{Laboratoire Modal'X, Universit\'e Paris 10, Nanterre, France}
\email{nmarie@u-paris10.fr}
\address{Laboratoire ISTI, ESME Sudria, Paris, France}
\email{marie@esme.fr}
\keywords{Financial model ; Stochastic volatility ; Fractional Brownian motion ; Viability condition.}
\date{}
\maketitle
%


%
\begin{abstract}
This paper deals with an extension of the so-called Black-Scholes model in which the volatility is modeled by a linear combination of the components of the solution of a differential equation driven by a fractional Brownian motion of Hurst parameter greater than $1/4$. In order to ensure the positiveness of the volatility, the coefficients of that equation satisfy a viability condition. The absence of arbitrages, the completeness of the market and a pricing formula are established.
\end{abstract}
\tableofcontents
\noindent
\textbf{MSC2010 :} 60H10, 91G80.
\\
\\
\textbf{Acknowledgments.} Many thanks to Wassim Hamra, who read this work during his MSc thesis and made some constructive comments.
%


%
\section{Introduction}
This paper deals with an extension of the so-called Black-Scholes model in which the volatility is modeled by a linear combination of the components of the solution of a differential equation driven by a fractional Brownian motion $B$ of Hurst parameter $H\in ]1/4,1[$. Precisely, the volatility is modeled by the $d$-dimensional ($d\in\mathbb N\backslash\{0,1\}$) process $V := (V(t))_{t\in [0,T]}$ such that
\begin{displaymath}
V(t) :=
\sum_{k = 1}^{d}\langle h_k,U(t)\rangle e_k
\end{displaymath}
where, $h_1,\dots,h_d\in\mathbb R^d\backslash\{0\}$ and $U := (U(t))_{t\in [0,T]}$ is the solution of the pathwise differential equation
\begin{displaymath}
U(\omega,t) =
u_0 +\int_{0}^{t}\mu(\omega,U(\omega,s))ds +\int_{0}^{t}\sigma(\omega,U(\omega,s))\textrm{ $\circ$}dB(\omega,s).
\end{displaymath}
The conditions on $\mu$ and $\sigma$ are stated at Section 2 and Appendix B, and are in accordance with the pathwise stochastic calculus framework (see Friz and Victoir \cite{FV10}).
\\
\\
For financial applications, the components of the process $V$ should be positive. There exists several ways to ensure the positiveness of the solution of a differential equation. One of them consists to choose $\mu$ and $\sigma$ having a singular behavior around zero, and to prove that the components of the solution stay positive because they cannot hit zero. About these models, in It\^o's calculus framework, see Karlin and Taylor \cite{KT81}, and in the pathwise stochastic calculus framework, see for instance Marie \cite{MARIE15_singular}. A famous example is the Cox-Ingersoll-Ross equation, which models the volatility in the Heston model for instance. About the Heston model, in It\^o's calculus framework see Heston \cite{HESTON93}, and in the pathwise stochastic calculus framework, see for instance Comte et al. \cite{CCR12} or Marie \cite{MARIE15_singular}. Another way to constrain the solution of a differential equation, which is used in this paper, is to assume that $\mu$ and $\sigma$ satisfy a viability condition. For ordinary differential equations, see Aubin \cite{AUBIN90} and Aubin et al. \cite{ABS11}, in It\^o's calculus framework, see the seminal paper \cite{AD90} of J.P. Aubin and G. DaPrato, and in the pathwise stochastic calculus framework, see Ciotir and Rascanu \cite{CR08} and Coutin and Marie \cite{CM16}. The main results of Coutin and Marie \cite{CM16} are summarized at the end of Appendix B.
\\
\\
Fractional market models have been studied by several authors as L.C.G. Rogers, who proved in \cite{ROGERS97} the existence of arbitrages in these models, or P. Cheridito who suggested in \cite{CHERIDITO01} to use the mixed fractional Brownian motion in order to bypass that difficulty. In this paper, as presented at Section 2, the fractional Brownian motion involves only in the equation which models the volatility. So, via some specific assumptions stated at Section 3, it is possible to ensure the absence of arbitrages and the completeness of the market, and to prove a pricing formula. The market is complete under a leverage condition as in Vilela Mendes et al. \cite{VOR15}.
\\
\\
Empirically, as established in Gatheral et al. \cite{GJR14}, the volatility process has $\alpha$-H\"older continuous paths with $\alpha\in ]0,1/2[$. That is taken into account by the model studied in this paper because in the pathwise stochastic calculus framework, the memory of the solution and the regularity of its paths are controlled by the Hurst parameter of the fractional Brownian motion which can belong to $]1/4,1/2[$ here.
\\
\\
Some simulations of the model are provided at Section 4.
\\
\\
The following notations are used throughout the paper. Some specific notations, related to the fractional Brownian motion and to the pathwise stochastic calculus, are introduced at appendices A and B respectively. 
\\
\\
\textbf{Notations :} Consider $d,e\in\mathbb N^*$.
\begin{enumerate}
 \item The euclidean scalar product (resp. norm) on $\mathbb R^d$ is denoted by $\langle .,.\rangle$ (resp. $\|.\|$). For every $x\in\mathbb R^d$, its $j$-th coordinate with respect to the canonical basis $(e_k)_{k\in\llbracket 1,d\rrbracket}$ of $\mathbb R^d$ is denoted by $x^{(j)}$ for every $j\in\llbracket 1,d\rrbracket$.
 \item For every $a,v\in\mathbb R^d$, $\Pi(a,v)$ denotes the half-space of $\mathbb R^d$ defined by $a$ and $v$ :
 \begin{displaymath}
 \Pi(a,v) :=
 \{x\in\mathbb R^d :\langle v,x - a\rangle\leqslant 0\}.
 \end{displaymath}
 \item The space of the matrices of size $d\times e$ is denoted by $\mathcal M_{d,e}(\mathbb R)$. For every $M\in\mathcal M_{d,e}(\mathbb R)$, its $(i,j)$-th coordinate with respect to the canonical basis $(e_{k,l})_{(k,l)\in\llbracket 1,d\rrbracket\times \llbracket 1,e\rrbracket}$ of $\mathcal M_{d,e}(\mathbb R)$ is denoted by $M_{i,j}$ for every $(i,j)\in\llbracket 1,d\rrbracket\times \llbracket 1,e\rrbracket$.
 \item The space of the continuous functions from $[0,T]$ into $\mathbb R^d$ is denoted by $C^0([0,T],\mathbb R^d)$ and equipped with the uniform norm $\|.\|_{\infty,T}$ such that :
 \begin{displaymath}
 \|f\|_{\infty,T} :=
 \sup_{t\in [0,T]}
 \|f(t)\|
 \textrm{ $;$ }
 \forall f\in C^0([0,T],\mathbb R^d).
 \end{displaymath}
 \item The set of all the dissections of an interval $I$ of $\mathbb R_+$ is denoted by $\mathfrak D_I$.
 \item Consider $p\in [1,\infty[$ and $s,t\in [0,T]$ such that $s < t$. A continuous function $x : [s,t]\rightarrow\mathbb R^d$ is of finite $p$-variation if and only if,
 \begin{eqnarray*}
  \|x\|_{p\textrm{-var},s,t} & := &
  \sup\left\{\left|
  \sum_{k = 1}^{n - 1}\|x(t_{k + 1}) - x(t_k)\|^p\right|^{1/p}\textrm{$;$ }
  n\in\mathbb N^*\textrm{ and }
  (t_k)_{k\in\llbracket 1,n\rrbracket}\in\mathfrak D_{[s,t]}\right\}\\
  & < & \infty.
 \end{eqnarray*}
 Consider the vector space
 \begin{displaymath}
 C^{p\textrm{-var}}([s,t],\mathbb R^d) :=
 \{x\in C^0([s,t],\mathbb R^d) :\|x\|_{p\textrm{-var},s,t} <\infty\}.
 \end{displaymath}
 The map $\|.\|_{p\textrm{-var},s,t}$ is a semi-norm on $C^{p\textrm{-var}}([s,t],\mathbb R^d)$.
 \item Consider a filtered probability space $(\Omega,\mathcal F,\mathbb F,\mathbb P)$, and a $\mathbb F$-martingale $M := (M(t))_{t\in [0,T]}$. The exponential martingale associated to $M$ and $q > 0$ is denoted by $\mathcal E^q(M)$ :
 \begin{displaymath}
 \mathcal E^q(M)(t) :=
 \exp\left(-\frac{q^2}{2}\langle M\rangle_t + qM(t)\right)
 \textrm{ $;$ }
 \forall t\in [0,T].
 \end{displaymath}
 In particular, $\mathcal E(M) :=\mathcal E^1(M)$.
 \item The space of the $d$-dimensional processes $Z := (Z(t))_{t\in [0,T]}$, adapted to $\mathbb F$, with continuous paths, and such that
 \begin{displaymath}
 \mathbb E\left(\int_{0}^{T}\|Z(s)\|^2ds\right)
 <\infty
 \end{displaymath}
 is denoted by $\mathbb H^2$.
\end{enumerate}
%


%
\section{The model}
This section introduces the multidimensional fractional stochastic volatility model studied in this paper. Consider a filtered probability space $(\Omega,\mathcal F,\mathbb F,\mathbb P)$ with $\mathcal F_T =\mathcal F$. The first subsection deals with the pathwise differential equation which models the volatility of the assets and is driven by $B := (B(t))_{t\in [0,T]}$, a $d$-dimensional fractional Brownian motion of Hurst parameter $H\in ]1/4,1[$ defined on $(\Omega,\mathcal F,\mathbb P)$ and adapted to $\mathbb F$. The second subsection deals with the stochastic differential equation which models the prices of the assets and is driven by $W := (W(t))_{t\in [0,T]}$, a $d$-dimensional Brownian motion defined on $(\Omega,\mathcal F,\mathbb P)$ and adapted to $\mathbb F$.
%


%
\subsection{The volatility of the assets}
Consider $\mu :\Omega\times\mathbb R^d\rightarrow\mathbb R^d$ and $\sigma :\Omega\times\mathbb R^d\rightarrow\mathcal M_d(\mathbb R)$, two maps satisfying assumptions \ref{existence_uniqueness_assumption} and \ref{integrability_assumption}. By Theorem \ref{existence_uniqueness_Young}, the pathwise differential equation
\begin{equation}\label{volatility_equation}
U(\omega,t) =
u_0 +\int_{0}^{t}\mu(\omega,U(\omega,s))ds +\int_{0}^{t}\sigma(\omega,U(\omega,s))\textrm{ $\circ$}dB(\omega,s)
\end{equation}
has a unique solution $U := (U(t))_{t\in [0,T]}$, and for an arbitrarily chosen $p > 1/H$, there exists a deterministic constant $C > 0$ such that for every $q > 0$,
\begin{equation}\label{control_volatility}
\|U\|_{\infty,T}
\leqslant
C\exp(CM(\mu,\sigma)(T^p + M_p(B)))
\in L^q(\Omega,\mathcal F,\mathbb P).
\end{equation}
Let $V := (V(t))_{t\in [0,T]}$ be the process defined by
\begin{displaymath}
V(t) :=\sum_{k = 1}^{d}\langle h_k,U(t)\rangle e_k
\textrm{ $;$ }
\forall t\in [0,T]
\end{displaymath}
where, $h_1,\dots,h_d\in\mathbb R^d\backslash\{0\}$. In order to model the volatility of the assets, the components of the process $V$ have to be positive, and it is true if and only if,
\begin{displaymath}
U(t)\in K :=\bigcap_{k = 1}^{d}\Pi(0,-h_k)
\textrm{ $;$ }
\forall t\in [0,T].
\end{displaymath}
The set $K$ is a convex polyhedron defined by the half-spaces $\Pi(0,-h_k)$ ; $k\in\llbracket 1,d\rrbracket$. Therefore, by Theorem \ref{viability_theorem}, the components of $V$ are positive if and only if $\mu$ and $\sigma$ satisfy the following assumption.
%


%
\begin{assumption}\label{viability_assumption_1}
For every $\omega\in\Omega$, $k\in\llbracket 1,d\rrbracket$ and $x\in\partial\Pi(0,-h_k)$,
\begin{displaymath}
\langle h_k,\mu(\omega,x)\rangle\geqslant 0
\end{displaymath}
and
\begin{displaymath}
\langle h_k,\sigma_{.,j}(\omega,x)\rangle = 0
\textrm{ $;$ }
\forall j\in\llbracket 1,d\rrbracket.
\end{displaymath}
\end{assumption}
\noindent
In order to model the volatility of the assets, the process $V$ has to be adapted to $\mathbb F$. Then, the maps $\mu$ and $\sigma$ have to satisfy the following assumption.
%


%
\begin{assumption}\label{measurability_assumption}
The maps $\mu(.,Z(.,t))$ and $\sigma(.,Z(.,t))$ are $\mathcal F_t$-measurable for every process $(Z(t))_{t\in [0,T]}$ adapted to $\mathbb F$ and $t\in [0,T]$.
\end{assumption}
\noindent
\textbf{Example.} Let $\xi :\Omega\rightarrow\mathbb R$ be a $\mathcal F_0$-measurable random variable, and assume there exists $\mu^* :\mathbb R\times\mathbb R^d\rightarrow\mathbb R^d$ and $\sigma^* :\mathbb R\times\mathbb R^d\rightarrow\mathcal M_d(\mathbb R)$ such that
\begin{displaymath}
\mu(\omega,y) =\mu^*(\xi(\omega),y)
\textrm{ and }
\sigma(\omega,y) =\sigma^*(\xi(\omega),y)
\end{displaymath}
for every $\omega\in\Omega$ and $y\in\mathbb R^d$. If the maps $\mu^*(x,.)$ and $\sigma^*(x,.)$ are continuous for every $x\in\mathbb R$, then $\mu$ and $\sigma$ satisfy Assumption \ref{measurability_assumption}.
\\
\\
Under assumptions \ref{existence_uniqueness_assumption}, \ref{integrability_assumption} and \ref{measurability_assumption}, $V\in\mathbb H^2$. Indeed, $V$ is adapted to $\mathbb F$ because the It\^o map associated to Equation (\ref{volatility_equation}) is continuous with respect to the driving signal, and by Equation \mbox{(\ref{control_volatility}) :}
\begin{displaymath}
\mathbb E\left(
\int_{0}^{T}|V(s)^{(k)}|^2ds\right)
\leqslant T\|h_k\|^2\mathbb E(\|U\|_{\infty,T}^{2}) <\infty
\end{displaymath}
for every $k\in\llbracket 1,d\rrbracket$.
%


%
\subsection{The prices of the assets}
The process $S_0 := (S_0(t))_{t\in [0,T]}$ of the prices of the risk-free asset is the solution of the ordinary differential equation
\begin{equation}\label{risk_free_equation}
S^0(t) = S_{0}^{0} + r\int_{0}^{t}S^0(u)du
\end{equation}
with $r > 0$. The process $S := (S(t))_{t\in [0,T]}$ of the prices of the risky assets is the solution of the $d$-dimensional stochastic differential equation
\begin{equation}\label{risky_equation}
S(t) =
S_0 +\int_{0}^{t}b(S(u))du +
\int_{0}^{t}f(S(u),U(u))dW(u)
\end{equation}
where, $b :\mathbb R^d\rightarrow\mathbb R^d$ is the map defined by
\begin{displaymath}
b(x) :=
\sum_{k = 1}^{d}
b_kx^{(k)}e_k
\textrm{ $;$ }
\forall x\in\mathbb R^d
\end{displaymath}
with $b_1,\dots,b_d\in\mathbb R^d$, and $f :\mathbb R^d\times\mathbb R^d\rightarrow\mathcal M_d(\mathbb R)$ is the map defined by
\begin{displaymath}
f(x,y) :=
\sum_{k = 1}^{d}x^{(k)}
\langle h_k,y\rangle e_{k,k}
\textrm{ $;$ }
\forall (x,y)\in\mathbb R^d\times\mathbb R^d.
\end{displaymath}
For every $t\in [0,T]$ and $k\in\llbracket 1,d\rrbracket$, by It\^o's formula :
\begin{equation}\label{Black_Scholes_solution}
S(t)^{(k)} =
S_{0}^{(k)}
\exp\left(
\int_{0}^{t}\left(b_k -\frac{1}{2}
|V(s)^{(k)}|^2\right)ds +
\int_{0}^{t}V(s)^{(k)}dW(s)^{(k)}
\right).
\end{equation}
%


%
\section{No-arbitrage, completeness and pricing formula}
This section deals with the viability and the completeness of the financial market modeled by Equation (\ref{volatility_equation})-(\ref{risk_free_equation})-(\ref{risky_equation}). On the one hand, in order to show the existence of a risk-neutral probability measure via Girsanov's theorem, Assumption \ref{viability_assumption_1} has to be improved. On the other hand, in order to get the completeness of the market, the volatility and the prices of the assets have to depend on the same source of randomness (see Vilela Mendes et al. \cite{VOR15}). In order to show the uniqueness of the risk-neutral probability measure, the fractional Brownian motion $B$ will be defined with respect to the Brownian motion $W$ via the Decreusefond-Ust\"unel transformation (see Appendix A).
%


%
\subsection{Existence of a risk-neutral probability measure}
In order to show the existence of a risk-neutral probability measure, which ensures the crucial no-arbitrage condition of the market, the behavior of $V$ around $0$ has to be controlled.
\\
\\
The components of the process $V$ are strictly positive if and only if, for every $\omega\in\Omega$, there exists $\xi(\omega) > 0$ such that for every $k\in\llbracket 1,d\rrbracket$ and $t\in [0,T]$, $V(\omega,t)^{(k)}\geqslant\xi(\omega)$. By construction, that is true when
\begin{displaymath}
U(\omega,t)\in K(\omega) :=
\bigcap_{k = 1}^{d}\Pi\left(
\frac{\xi(\omega)}{h_{k}^{(i_k)}}e_{i_k},-h_k\right)
\textrm{$;$ }
\forall\omega\in\Omega
\textrm{$,$ }
\forall t\in [0,T]
\end{displaymath}
where, for every $k\in\llbracket 1,d\rrbracket$, $i_k\in\llbracket 1,d\rrbracket$ and $h_{k}^{(i_k)}\not= 0$. For every $\omega\in\Omega$, $K(\omega)$ is a convex polyhedron. Therefore, by Theorem \ref{viability_theorem}, the components of $V$ are positive when $\mu$ and $\sigma$ satisfy the following assumption.
%


%
\begin{assumption}\label{viability_assumption_2}
For every $\omega\in\Omega$, $k\in\llbracket 1,d\rrbracket$ and $x\in\partial\Pi(\xi(\omega)/h_{k}^{(i_k)}e_{i_k},-h_k)$,
\begin{displaymath}
\langle h_k,\mu(\omega,x)\rangle\geqslant 0
\end{displaymath}
and
\begin{displaymath}
\langle h_k,\sigma_{.,j}(\omega,x)\rangle = 0
\textrm{ $;$ }
\forall j\in\llbracket 1,d\rrbracket.
\end{displaymath}
\end{assumption}
\noindent
In the sequel, the map $\xi :\Omega\rightarrow\mathbb R_{+}^{*}$ satisfies the following assumption.
%


%
\begin{assumption}\label{integrability_xi}
$\xi$ is $\mathcal F_0$-measurable, $\xi(\Omega)\subset [0,m]$ with $m > 0$, and for every $q > 0$,
\begin{displaymath}
\mathbb E\left(\exp\left(\frac{q}{\xi^2}\right)\right) <\infty.
\end{displaymath}
\end{assumption}
\noindent
\textbf{Examples :}
\begin{enumerate}
 \item If $\xi$ is constant, then it satisfies Assumption \ref{integrability_xi}.
 \item Assume that
 \begin{displaymath}
 \left.\frac{d\mathbb P_{\xi}}{dx}\right|_{\mathcal F_0} =
 \lambda_1\exp\left(
 -\frac{\lambda_2}{x^n}\right)\mathbf 1_{]0,\lambda_3]}(x)
 \end{displaymath}
 with $n\in\mathbb N\backslash\{0,1\}$, and $\lambda_1,\lambda_2,\lambda_3 > 0$ satisfying
 \begin{displaymath}
 \lambda_1\int_{0}^{\lambda_3}\exp\left(
 -\frac{\lambda_2}{x^n}\right)dx = 1.
 \end{displaymath}
 Let $q > 0$ be arbitrarily chosen. By the transfer theorem :
 \begin{displaymath}
 \mathbb E\left(
 \exp\left(\frac{q}{\xi^2}\right)\right) =
 \lambda_1\int_{1/\lambda_3}^{\infty}
 e^{qx^2 -\lambda_2x^n}\frac{dx}{x^2}.
 \end{displaymath}
 So, if $n > 2$, then $\xi$ satisfies Assumption \ref{integrability_xi}.
 \\
 \\
 That distribution is interesting because it curbs the components of $V$ if and only if they go under the level $\lambda_3$.
\end{enumerate}
%


%
\begin{proposition}\label{no_arbitrage}
Under assumptions \ref{existence_uniqueness_assumption}, \ref{integrability_assumption}, \ref{measurability_assumption}, \ref{viability_assumption_2} and \ref{integrability_xi}, there exists at least one risk-neutral probability measure $\mathbb P^*\sim\mathbb P$.
\end{proposition}
%


%
\begin{proof}
Consider the $d$-dimensional process $\theta := (\theta(t))_{t\in [0,T]}$ such that :
\begin{displaymath}
\theta(t) :=
\sum_{k = 1}^{d}
\frac{r - b_k}{V(t)^{(k)}}e_k
\textrm{ $;$ }
\forall t\in [0,T].
\end{displaymath}
Since $V\in\mathbb H^2$ (see Subsection 2.1), $\theta$ is adapted to $\mathbb F$. By Assumption \ref{integrability_xi} :
\begin{eqnarray*}
 \mathbb E
 \left(\exp\left(\frac{1}{2}\int_{0}^{T}\|\theta(s)\|^2ds\right)\right) & = &
 \mathbb E\left(\exp\left(
 \frac{1}{2}\sum_{k = 1}^{d}\int_{0}^{T}
 \left|\frac{r - b_k}{V(s)^{(k)}}\right|^2ds\right)\right)\\
 & \leqslant &
 \mathbb E\left(
 \exp\left(\frac{d\beta T}{2\xi^2}\right)\right) <\infty
\end{eqnarray*}
where,
\begin{displaymath}
\beta :=
\max_{k\in\llbracket 1,d\rrbracket}
|r - b_k|^2.
\end{displaymath}
So, by Girsanov's theorem (see Bj\"ork \cite{BJORK09}, Theorem 11.3), under the probability measure $\mathbb P^*$ such that
\begin{displaymath}
\left.\frac{d\mathbb P^*}{d\mathbb P}\right|_{\mathcal F_T} =
\mathcal E\left(
\int_{0}^{.}\langle\theta(s),dW(s)\rangle\right)(T),
\end{displaymath}
the following process is a Brownian motion adapted to $\mathbb F$ :
\begin{displaymath}
W^* := W -\int_{0}^{.}\theta(s)ds.
\end{displaymath}
Let $\widetilde S := S/S^0$ be the discounted prices process of the risky assets. By It\^o's \mbox{formula :}
\begin{displaymath}
\widetilde S(t) =
\widetilde S(0) +\int_{0}^{t}f(\widetilde S(u),U(u))dW^*(u).
\end{displaymath}
Therefore, under $\mathbb P^*$, since $W^*$ is a Brownian motion adapted to $\mathbb F$, $\widetilde S$ is a $\mathbb F$-martingale.
\end{proof}
%


%
\subsection{Leverage, completeness and pricing formula}
As mentioned in Vilela Mendes et al. \cite{VOR15}, a way to get the completeness of a market modeled by a fractional stochastic volatility model and a pricing formula is to consider the fractional Brownian motion $B$ associated to $W$ by the Decreusefond-Ust\"unel transformation (see Appendix A). There is a leverage effect.
\\
\\
Let $(\Omega,\mathcal F,\mathbb P)$ be the canonical probability space for $W$, and consider the process $B := (B(t))_{t\in [0,T]}$ defined by
\begin{displaymath}
B(t) :=
\int_{0}^{t}K_H(t,s)dW(s)
\textrm{ $;$ }
\forall t\in [0,T].
\end{displaymath}
By Theorem \ref{DU_representation}, $B$ is a fractional Brownian motion of Hurst parameter $H$ generating the same filtration $\mathbb F$ than $W$.
\\
\\
In this subsection, the probability space $(\Omega,\mathcal F,\mathbb P)$ is equipped with the filtration $\mathbb F$ generated by $W$ and $B$. That is crucial in order to apply Girsanov's converse (see Bj\"ork \cite{BJORK09}, Theorem 11.6) to get the completeness of the market modeled by (\ref{volatility_equation})-(\ref{risk_free_equation})-(\ref{risky_equation}).
%


%
\begin{proposition}\label{completeness}
Under assumptions \ref{existence_uniqueness_assumption}, \ref{integrability_assumption}, \ref{measurability_assumption}, \ref{viability_assumption_2} and \ref{integrability_xi}, there exists a unique risk-neutral probability measure $\mathbb P^*\sim\mathbb P$.
\end{proposition}
%


%
\begin{proof}
As established at Proposition \ref{no_arbitrage}, the probability measure $\mathbb P^*$ such that
\begin{displaymath}
\left.\frac{d\mathbb P^*}{d\mathbb P}\right|_{\mathcal F_T} =
\mathcal E\left(
\int_{0}^{.}\langle\theta(s),dW(s)\rangle\right)(T)
\end{displaymath}
is risk-neutral. Assume there exists another probability measure $\mathbb P^{**}\sim\mathbb P$ such that $\widetilde S$ is a $\mathbb F$-martingale. By Girsanov's converse (see Bj\"ork \cite{BJORK09}, Theorem 11.6), there exists a $d$-dimensional process $\psi := (\psi(t))_{t\in [0,T]}$, adapted to $\mathbb F$, such that
\begin{equation}\label{completeness_1}
W^{**} :=
W -\int_{0}^{.}\psi(s)ds
\end{equation}
is a Brownian motion adapted to $\mathbb F$, and
\begin{equation}\label{completeness_2}
\left.\frac{d\mathbb P^{**}}{d\mathbb P}\right|_{\mathcal F_t} =
\mathcal E\left(\int_{0}^{.}
\langle\psi(s),dW(s)\rangle\right)(t)
\textrm{ $;$ }
\forall t\in [0,T].
\end{equation}
Let $k\in\llbracket 1,d\rrbracket$ be arbitrarily chosen. By Equation (\ref{completeness_1}) :
\begin{eqnarray*}
 \widetilde S(t)^{(k)} & = &
 \widetilde S(0)^{(k)} +
 \int_{0}^{t}
 (b_k - r + V(u)^{(k)}\psi(u)^{(k)})\widetilde S(u)^{(k)}du +\\
 & &
 \int_{0}^{t}
 V(u)^{(k)}\widetilde S(u)^{(k)}dW^{**}(u)^{(k)}.
\end{eqnarray*}
So, since $\widetilde S$ is a $\mathbb F$-martingale under $\mathbb P^{**}$ :
\begin{displaymath}
b_k - r + V(t)^{(k)}\psi(t)^{(k)} = 0
\textrm{ $;$ }
\forall t\in [0,T].
\end{displaymath}
In other words, $\psi =\theta$. Therefore, $\mathbb P^{**} =\mathbb P^*$ by Equation (\ref{completeness_2}).
\end{proof}
%


%
\begin{proposition}\label{pricing_formula}
Consider a $\mathcal F_T$-measurable random variable $h\in L^2(\Omega,\mathcal F,\mathbb P^*)$. There exists an admissible strategy $\varphi := (H^0(t),H(t))_{t\in [0,T]}$ such that
\begin{eqnarray*}
 V(\varphi,t) & := &
 H^0(t)S^0(t) + H(t)S(t)\\
 & = &
 \mathbb E^*(e^{-r(T - t)}h|\mathcal F_t)
\end{eqnarray*}
for every $t\in [0,T]$.
\end{proposition}
%


%
\begin{proof}
The proof is the same than for the classic Black-Scholes model. Indeed, since the processes $W$ and $B$ generate the same filtration, the key argument of the proof using the martingale representation theorem (Revuz and Yor \cite{RY99}, Theorem V.3.9) holds true.
\end{proof}
%


%
\section{Numerical simulations}
The purpose of this section is to provide some simulations of the fractional stochastic volatility model studied in this paper when $H > 1/2$.
\\
\\
The fractional Brownian motion $B$ is simulated via Wood-Chan's method (see Coeurjolly \cite{COEURJOLLY00}, Section 3.6). The solution of Equation (\ref{volatility_equation}) is approximated by the (explicit) Euler scheme (see Lejay \cite{LEJAY10}, Section 5).
\\
\\
Assume that $d := 2$, $h_1 := (1,1)$, $h_2 := (1,0)$, $i_1 = i_2 = 1$, and
\begin{displaymath}
\left.
\frac{d\mathbb P_{\xi}}{dx}\right|_{\mathcal F_0} =
\lambda\exp\left(-\frac{1}{x^3}\right)\mathbf 1_{]0,1]}(x)
\end{displaymath}
with
\begin{displaymath}
\lambda :=
\left(\int_{0}^{1}\exp\left(-\frac{1}{x^3}\right)dx\right)^{-1}
\approx 15.7604.
\end{displaymath}
So, for every $\omega\in\Omega$,
\begin{displaymath}
K(\omega) =
\{(x,y)\in\mathbb R^2 :
\xi(\omega) - y\leqslant x
\textrm{ and }
\xi(\omega)\leqslant x\}.
\end{displaymath}
Consider the maps $\mu :\Omega\times\mathbb R^2\rightarrow\mathbb R^2$ and $\sigma :\Omega\times\mathbb R^2\rightarrow\mathcal M_2(\mathbb R)$ defined by
\begin{displaymath}
\mu(\omega,(x,y)) :=
xe_1 + (y -\xi(\omega))e_2
\end{displaymath}
and
\begin{displaymath}
\sigma(\omega,(x,y)) :=
(x -\xi(\omega))
\begin{pmatrix}
 1 & 1\\
 0 & -1
\end{pmatrix}
\end{displaymath}
for every $\omega\in\Omega$ and $(x,y)\in\mathbb R^2$. These maps satisfy assumptions \ref{existence_uniqueness_assumption}, \ref{integrability_assumption}, \ref{measurability_assumption} and \ref{viability_assumption_2}.
\\
\\
According with Section 2, the volatility of the assets is modeled by the process $V := (V(t))_{t\in [0,T]}$ such that
\begin{displaymath}
V(t) := (U(t)^{(1)} + U(t)^{(2)},U(t)^{(1)})
\textrm{ $;$ }
\forall t\in [0,T]
\end{displaymath}
where, the process $U := (U(t))_{t\in [0,T]}$ is the solution of the pathwise differential equation
\begin{displaymath}
U(\omega,t) = (1,0) +
\int_{0}^{t}\mu(\omega,U(\omega,s))ds +
\int_{0}^{t}\sigma(\omega,U(\omega,s))\textrm{ $\circ$}dB(\omega,s),
\end{displaymath}
and the process $S := (S(t))_{t\in [0,T]}$ of the prices of the risky assets is the solution of the stochastic differential equations
\begin{eqnarray*}
 S(t)^{(1)} & = & 1 +
 \int_{0}^{t}S(u)^{(1)}du +\int_{0}^{t}S(u)^{(1)}(U(u)^{(1)} + U(u)^{(2)})dW(u)^{(1)},\\
 S(t)^{(2)} & = & 1 +
 \int_{0}^{t}S(u)^{(2)}du +\int_{0}^{t}S(u)^{(2)}U(u)^{(1)}dW(u)^{(2)}.
\end{eqnarray*}
With $H := 0.7$, some paths of $U$, $V$, $S^{(1)}$ and $S^{(2)}$ are plotted on the following figures. Note the viability of $U$ (resp. $V$) in $K$ (resp. $[\xi,\infty[^2$).
\begin{figure}[htbp]
\begin{center}
\includegraphics[scale = 0.50]{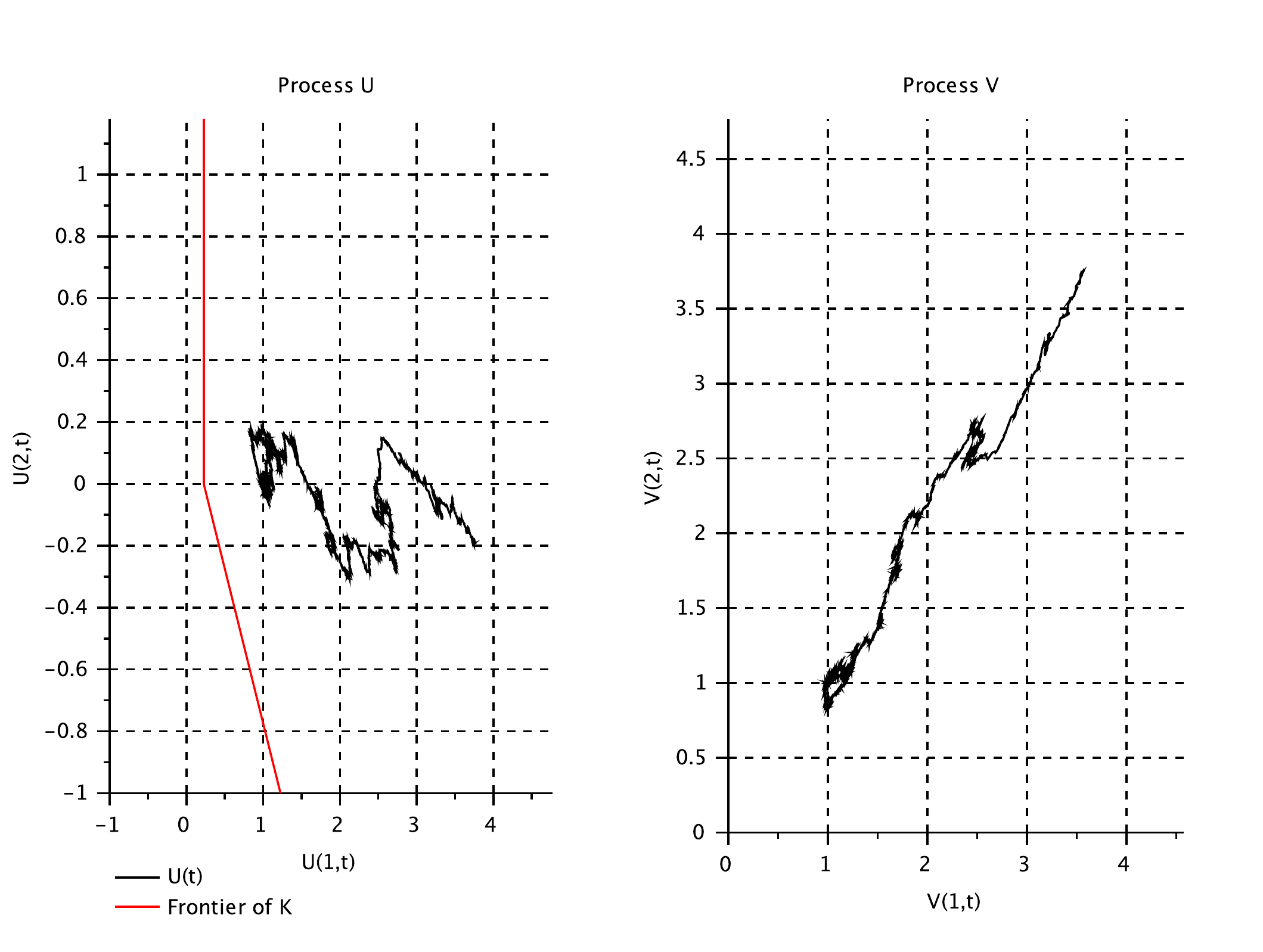}
\end{center}
\caption{Volatility process}
\label{fig:image1}
\end{figure}
\begin{figure}[htbp]
\begin{center}
\includegraphics[scale = 0.50]{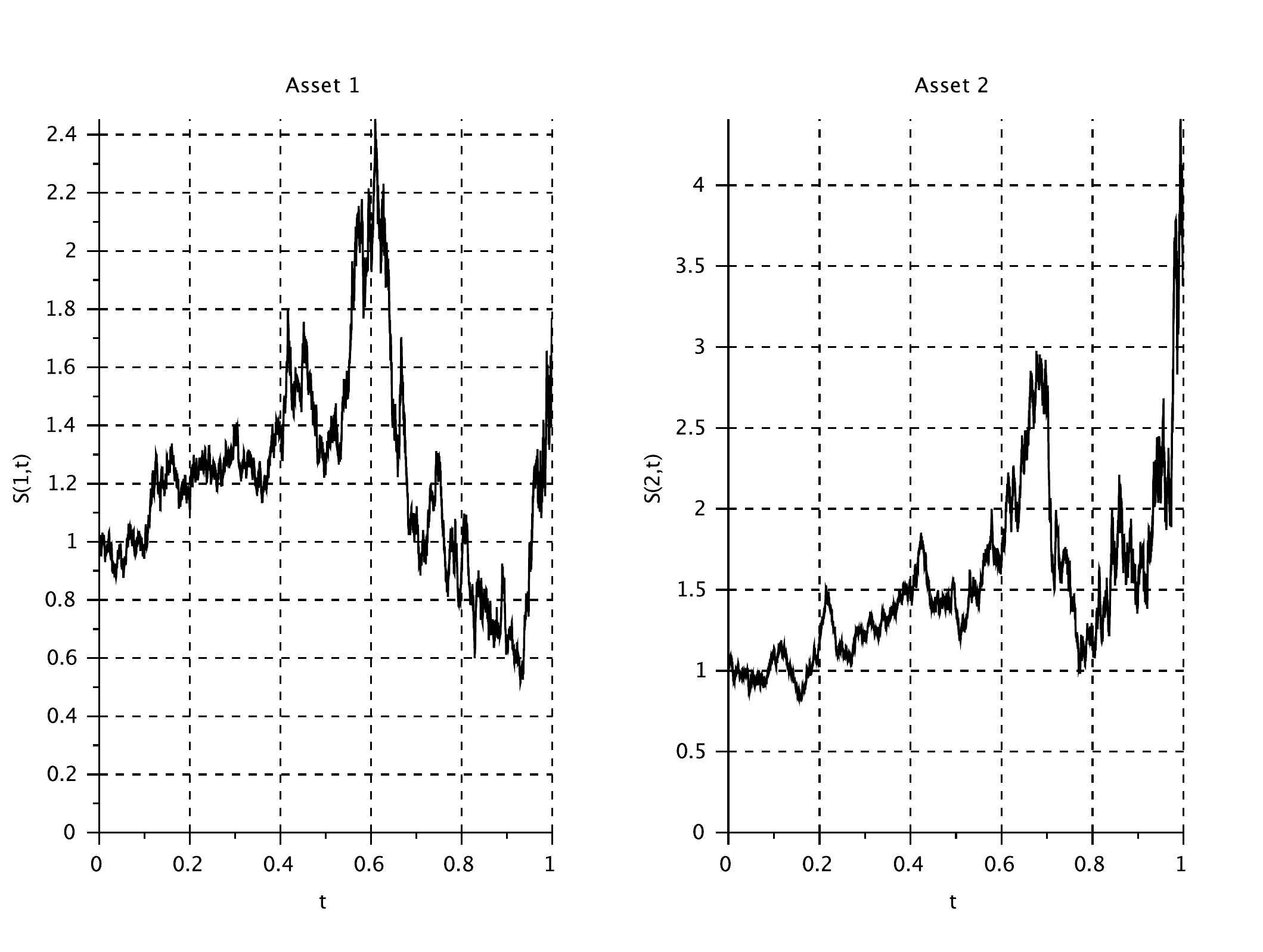}
\end{center}
\caption{Prices process}
\label{fig:image1}
\end{figure}
\newpage
\appendix
%


%
\section{The fractional Brownian motion}
This subsection presents the fractional Brownian motion and some results of the seminal paper of Decreusefond and Ust\"unel \cite{DU99}.
%


%
\begin{definition}\label{fbm}
Let $B := (B(t))_{t\in [0,T]}$ be a centered Gaussian process. It is a fractional Brownian motion if and only if there exists $H\in ]0,1[$, called Hurst parameter of $B$, such \mbox{that :}
\begin{displaymath}
\normalfont{\textrm{cov}}(B(s),B(t)) =
\frac{1}{2}(|s|^{2H} +
|t|^{2H} -
|t - s|^{2H})
\end{displaymath}
for every $(s,t)\in [0,T]^2$.
\end{definition}
%


%
\begin{proposition}\label{regularity_fbm}
Let $B$ be a fractional Brownian motion of Hurst parameter $H\in ]0,1[$. The paths of $B$ are continuous and of finite $p$-variation with $p\in ]1/H,\infty[$.
\end{proposition}
\noindent
See Nualart \cite{NUALART06}, Section 5.1.
\\
\\
Consider $H\in ]0,1[$, the Gauss hypergeometric function $\mathfrak F$, and the map $K_H : [0,T]^2\rightarrow\mathbb R$ defined by
\begin{displaymath}
K_H(t,s) :=
\frac{(t - s)^{H - 1/2}}{\Gamma(H + 1/2)}
\mathfrak F\left(
\frac{1}{2} - H,H -\frac{1}{2},H +\frac{1}{2},1 -\frac{t}{s}\right)
\mathbf 1_{[0,t[}(s)
\end{displaymath}
for every $(s,t)\in [0,T]^2$ such that $s < t$.
\\
\\
Let $W := (W(t))_{t\in [0,T]}$ be a Brownian motion, and consider the canonical probability space $(\Omega,\mathcal F,\mathbb P)$ for $W$. Consider also the process $B := (B(t))_{t\in [0,T]}$ defined by
\begin{displaymath}
B(t) :=
\int_{0}^{t}K_H(t,s)dW(s)
\end{displaymath}
for every $t\in [0,T]$.
%


%
\begin{theorem}\label{DU_representation}
The process $B$ is a fractional Brownian motion of Hurst parameter $H$ generating the same filtration than $W$.
\end{theorem}
\noindent
See Decreusefond and Ust\"unel \cite{DU99}, Corollary 3.1.
%


%
\section{Differential equations driven by a fractional Brownian motion}
This section deals with some basics on pathwise differential equations driven by a fractional Brownian motion of Hurst parameter greater than $1/4$. At the end of this section, some viability results proved in Coutin and Marie \cite{CM16} are stated.
\\
\\
Consider $(\Omega,\mathcal F,\mathbb P)$ a probability space, $\textrm B := (\textrm B(t))_{t\in [0,T]}$ a fractional Brownian motion of Hurst parameter $H\in ]1/4,1[$, and $B := (B_1,\dots,B_e)$ where, $B_1,\dots,B_e$ are $e\in\mathbb N^*$ independent copies of $\textrm B$.
\\
\\
Consider the differential equation
\begin{equation}\label{main_equation_appendix}
X(\omega,t) =
u_0 +\int_{0}^{t}\mu(\omega,X(\omega,s))ds +
\int_{0}^{t}\sigma(\omega,X(\omega,s))\textrm{ $\circ$}dB(\omega,s)
\end{equation}
where, $\mu$ (resp. $\sigma$) is a map from $\Omega\times\mathbb R^d$ into $\mathbb R^d$ (resp. $\mathcal M_{d,e}(\mathbb R)$).
%


%
\begin{definition}\label{solution_differential_equation}
In the sense of rough paths, a process $X := (X(t))_{t\in [0,T]}$ is a solution on $[0,T]$ of Equation (\ref{main_equation_appendix}) if and only if
\begin{displaymath}
\lim_{n\rightarrow\infty}
\|X_n - X\|_{\infty,T} = 0
\end{displaymath}
where, for every $n\in\mathbb N^*$, $X_n$ is a solution on $[0,T]$ of the ordinary differential equation
\begin{displaymath}
X_n(\omega,t) =
u_0 +\int_{0}^{t}\mu(\omega,X_n(\omega,s))ds +
\int_{0}^{t}\sigma(\omega,X_n(\omega,s))\textrm{ $\circ$}dB_n(\omega,s)
\end{displaymath}
and $(B_n)_{n\in\mathbb N^*}$ is sequence of linear approximations of $B$ such that
\begin{displaymath}
\lim_{n\rightarrow\infty}
\|B_n - B\|_{p\normalfont{\textrm{-var}},T} = 0
\end{displaymath}
with $p > 1/H$.
\end{definition}
%


%
\begin{assumption}\label{existence_uniqueness_assumption}
For every $\omega\in\Omega$,
\begin{enumerate}
 \item $\mu(\omega,.)\in C^{[1/H] + 1}(\mathbb R^d,\mathbb R^d)$ and $\sigma(\omega,.)\in C^{[1/H] + 1}(\mathbb R^d,\mathcal M_{d,e}(\mathbb R))$.
 \item For every $k\in\llbracket 1,d\rrbracket$, the maps $D^k\mu(\omega,.)$ and $D^k\sigma(\omega,.)$ are bounded.
 \item $\mu(\omega,.)$ (resp. $\sigma(\omega,.)$) is Lipschitz continuous from $\mathbb R^d$ into itself (resp. $\mathcal M_{d,e}(\mathbb R)$).
\end{enumerate}
\end{assumption}
\noindent
\textbf{Notation.} Under Assumption \ref{existence_uniqueness_assumption}, for every $\omega\in\Omega$,
\begin{displaymath}
M(\mu,\sigma)(\omega) :=
\max_{k\in\llbracket 1,[1/H] + 1\rrbracket}
\|D^k(\mu(\omega,.),\sigma(\omega,.))\|_{\infty,\mathcal L((\mathbb R^d)^{\otimes k},\mathbb R^d\oplus\mathcal M_{d,e}(\mathbb R))}.
\end{displaymath}
%


%
\begin{assumption}\label{integrability_assumption}
For every $x\in\mathbb R^d$, $\mu(.,x)$ and $\sigma(.,x)$ are $\mathcal F$-measurable. Moreover, $M(\mu,\sigma)$ is bounded by a deterministic constant.
\end{assumption}
%


%
\begin{theorem}\label{existence_uniqueness_Young}
Under Assumption \ref{existence_uniqueness_assumption}, Equation (\ref{main_equation_appendix}) has a unique solution $X$ on $[0,T]$. Moreover, $X$ has continuous paths of finite $p$-variation with $p > 1/H$, and there exists a deterministic constant $C > 0$ such that
\begin{equation}\label{control_solution_Young}
\|X\|_{\infty,T}
\leqslant
C\exp(CM(\mu,\sigma)(T^p + M_p(B)))
\end{equation}
where,
\begin{eqnarray*}
 M_p(B) & := &
 \sup\{
 \sum_{k = 1}^{n - 1}\|B\|_{p\normalfont{\textrm{-var}},t_k,t_{k + 1}}^{p}
 \textrm{ $;$ }
 n\in\mathbb N^*
 \textrm{$,$ }
 (t_k)_{k\in\llbracket 1,n\rrbracket}
 \in\mathfrak D_{[0,T]}\\
 & &
 \textrm{ and }
 \forall k\in\llbracket 1,n - 1\rrbracket
 \textrm{$,$ }
 \|B\|_{p\normalfont{\textrm{-var}},t_k,t_{k + 1}}
 \leqslant 1\}.
\end{eqnarray*}
If $\mu$ and $\sigma$ satisfy Assumption \ref{integrability_assumption}, then
\begin{equation}\label{integrability_Young}
\mathbb E(\|X\|_{\infty,T}^{q}) <\infty
\textrm{ $;$ }
\forall q\in [1,\infty[.
\end{equation}
\end{theorem}
\noindent
\textbf{Remark.} At Theorem \ref{existence_uniqueness_Young}, the existence and the uniqueness of the solution of Equation (\ref{main_equation_appendix}), and the bound provided by Equation (\ref{control_solution_Young}), are given by Friz and Victoir \cite{FV10}, Theorem 10.26, Exercice 10.55 and Exercice 10.56. Equation (\ref{integrability_Young}) of Theorem \ref{existence_uniqueness_Young} is a straightforward consequence of Assumption \ref{integrability_assumption}, Equation (\ref{control_solution_Young}) and Cass, Lyons and Litterer \cite{CLL13}.
\\
\\
The following viability results are proved in Coutin and Marie \cite{CM16}. Let $K\subset\mathbb R^d$ be a closed convex set.
%


%
\begin{definition}\label{viability}
A function $\varphi : [0,T]\rightarrow\mathbb R^d$ is viable in $K$ if and only if,
\begin{displaymath}
\varphi(t)\in K
\textrm{ $;$ }
\forall t\in [0,T].
\end{displaymath}
\end{definition}
\noindent
\textbf{Notation.} For every $x\in K$,
\begin{displaymath}
N_K(x) :=
\{s\in\mathbb R^d :
\forall y\in K
\textrm{$,$ }
\langle s,y - x\rangle\leqslant 0\}
\end{displaymath}
is the normal cone to $K$ at $x$.
\begin{assumption}\label{viability_assumption}
For every $\omega\in\Omega$, $x\in\partial K$ and $s\in N_K(x)$,
\begin{displaymath}
\langle s,b(\omega,x)\rangle\leqslant 0
\end{displaymath}
and
\begin{displaymath}
\langle s,\sigma_{.,k}(\omega,x)\rangle\leqslant 0
\textrm{ $;$ }
\forall k\in\llbracket 1,e\rrbracket.
\end{displaymath}
\end{assumption}
%


%
\begin{theorem}\label{viability_theorem}
Under assumptions \ref{existence_uniqueness_assumption} and \ref{viability_assumption}, $X(\omega)$ is viable in $K$ for every $\omega\in\Omega$.
\end{theorem}
\noindent
See Coutin and Marie \cite{CM16}, Theorem 3.4.
\\
\\
\textbf{Remark.} Since Theorem \ref{viability_theorem} is a pathwise result, the viability of the solution in $K$ holds true when that one is a random set.
%


%

%
\end{document}